\documentclass[12pt]{article}

\usepackage{enumerate}
\usepackage{amsthm}
\usepackage{amssymb}
\usepackage{amsmath}
\usepackage{amsfonts}
\usepackage{cancel}
\usepackage{times}
\usepackage{bbm}
\usepackage{cite}
\usepackage{color, soul}
\usepackage{hyperref}
\hypersetup{
     colorlinks   = true,
     citecolor    = black,
     linkcolor    = blue
}

\usepackage[paper=a4paper, top=2.5cm, bottom=2.5cm, left=2.5cm, right=2.5cm]{geometry}

\def\qed{\hfill $\vcenter{\hrule height .3mm
\hbox {\vrule width .3mm height 2.1mm \kern 2mm \vrule width .3mm
height 2.1mm} \hrule height .3mm}$ \bigskip}

\def \RR {\mathbb R}
\def \NN {\mathbb N}
\def \EE {\mathbb E}

\def \PP {\mathbb P}

\def \vphi {\varphi}

\def \FF {\mathcal{F}}

\newcommand\norm[1]{\left\lVert#1\right\rVert}

\newtheorem{theorem}{Theorem}
\newtheorem{lemma}{Lemma}

\newtheorem{corollary}[theorem]{Corollary}
\theoremstyle{definition}

\theoremstyle{remark}
\newtheorem{remark}[theorem]{Remark}
\newtheorem*{remark*}{Remark}

\long\def\symbolfootnotetext[#1]#2{\begingroup
\def\thefootnote{\fnsymbol{footnote}}\footnotetext[#1]{#2}\endgroup}

\title{Stability of Talagrand's Gaussian transport-entropy inequality via the F\"ollmer process}
 
\author{Dan Mikulincer\thanks{Weizmann Institute of Science. Supported by an Azrieli foundation fellowship.}}

\begin{document}
	
\maketitle

\begin{abstract}
	We establish a dimension-free improvement of Talagrand's Gaussian transport-entropy inequality, under the assumption that the measures satisfy a Poincar\'e inequality. We also study stability of the inequality, in terms of relative entropy, when restricted to measures whose covariance matrix trace is smaller than the ambient dimension. In case the covariance matrix is strictly smaller than the identity, we give dimension-free estimates which depend on its eigenvalues. To complement our results, we show that our conditions cannot be relaxed, and that there exist measures with covariance larger than the identity, for which the inequality is not stable, in relative entropy. To deal with these examples, we show that, without any assumptions, one can always get quantitative stability estimates in terms of relative entropy to Gaussian mixtures. Our approach gives rise to a new point of view which sheds light on the hierarchy between Fisher information, entropy and transportation distance, and may be of independent interest. In particular, it implies that the described results apply verbatim to the log-Sobolev inequality and improve upon some known estimates in the literature. 

\end{abstract}

\section{Introduction}
Talagrand's Gaussian transport-entropy inequality, first proved in \cite{talagrand1996transportation}, states that for any measure $\mu$ in $\RR^d$, with a finite second moment matrix,
\begin{equation} \label{eq: talagrand}
\mathcal{W}_2^2\left(\mu,\gamma\right) \leq 2\mathrm{D}\left(\mu||\gamma\right).
\end{equation}
Here, $\gamma$ denotes the standard Gaussian measure on $\RR^d$, with density 
$$\vphi(x) = \frac{1}{\left(\sqrt{2\pi}\right)^d}e^{\frac{-\norm{x}_2^2}{2}}.$$
The distances involved in the inequality are,
$\mathrm{D}\left(\mu||\gamma\right)$ , the relative entropy, defined by
$$\mathrm{D}\left(\mu||\gamma\right) = \int\limits_{\RR^d}\ln\left(\frac{d\mu}{d\gamma}\right)d\mu,$$
and $\mathcal{W}_p\left(\mu,\gamma\right)$ is the $L_p$-Wasserstein distance (with $L_2$ cost function),
$$\mathcal{W}_p\left(\mu,\gamma\right) = \sqrt[p]{\inf\limits_{\pi}\int\limits_{\RR^{2d}}\norm{x-y}_2^pd\pi(x,y)},$$
where the infimum runs over all measures on $\RR^{2d}$ whose marginal laws onto the first and last $d$ coordinates are $\mu$ and $\gamma$. Since this fundamental inequality tensorizes, it holds in any dimension. Using this quality, the inequality was shown to imply a sharp form of the dimension-free concentration of measure phenomenon in Gaussian space. The reader is referred to \cite{villani2008optimal,gozlan2010transport,ledoux2001concentration} for further information on the topic.
By setting the measure $\mu$ to be a translation of $\gamma$, we can see that the inequality is tight and that, in particular, the constant $2$ in \eqref{eq: talagrand} cannot be improved. One, in fact, may show that these examples account for the only equality cases of \eqref{eq: talagrand}. We are thus led to consider the question of stability of the inequality.
Consider the deficit
$$\delta_{\mathrm{Tal}}(\mu) := 2\mathrm{D}(\mu||\gamma) - \mathcal{W}_2^2\left(\mu, \gamma\right).$$
Suppose that $\delta_{\mathrm{Tal}}(\mu)$ is small. In this case, must $\mu$ be necessarily close to a translate of $\gamma$?\\
 A first step towards answering this question, which serves as a starting point for the current work, was given in \cite{fathi2016quantitative} (see also \cite{kolesnikov2017moment}), where it was shown that there exists a numerical constant $c>0$, such that if $\mu$ is centered,
\begin{equation} \label{eq: fathi bound}
\delta_{\mathrm{Tal}}\left(\mu\right)\geq c\min\left(\frac{\mathcal{W}^2_{1,1}(\mu,\gamma)}{d}, \frac{\mathcal{W}_{1,1}(\mu,\gamma)}{\sqrt{d}}\right).
\end{equation}
Here, $\mathcal{W}_{1,1}$ stands for the $L_1$-Wasserstein distance with $L_1$-cost function.
The inequality was later improved in \cite{cordero2017transport}, and $\frac{\mathcal{W}_{1,1}(\mu,\gamma)}{\sqrt{d}}$ was replaced by the larger quantity $\mathcal{W}_1(\mu,\gamma)$.
One could hope to improve this result in several ways; First, one may consider stronger notions of distance than $\mathcal{W}_{1,1}$, like relative entropy. Indeed by Jensen's inequality and \eqref{eq: talagrand},
\begin{equation} \label{eq: wass1}
\frac{\mathcal{W}^2_{1,1}(\mu,\gamma)}{d} \leq \mathcal{W}_2^2(\mu,\gamma) \leq 2\mathrm{D}(\mu||\gamma).
\end{equation}
Second, note that for product measures, $\delta_{\mathrm{Tal}}(\mu)$ grows linearly in $d$, while the RHS of \eqref{eq: fathi bound} may grow like $\sqrt{d}$ (this remains true for the improved result, found in \cite{cordero2017transport}). The dimension-free nature of \eqref{eq: talagrand} suggests that the dependence on the dimension in \eqref{eq: fathi bound} should, hopefully, be removed. The goal of the present work is to identify cases in which \eqref{eq: fathi bound} may be improved. Specifically, we will be interested in giving dimension-free stability bounds with respect to the relative entropy distance. We will also show that, without further assumptions on the measure $\mu$, \eqref{eq: fathi bound} cannot be significantly improved.\\
 This work adds to a recent line of works which explored dimension-free stability estimates for functional inequalities in the Gaussian space, such as the log-Sobolev inequality \cite{eldan2019stability, fathi2016quantitative, ledoux2017stein, bobkov2014bounds, feo2017remarks}, the Shannon-Stam inequality \cite{courtade2018quantitative, eldan2019shannon} and the Gaussian isoperimetric inequality \cite{cianchi2011isoperimetric,mossel2015robust,barchiesi2017sharp}.
\subsection*{Results}
In our first main result, we restrict our attention to the subclass of probability measures which satisfy a Poincar\'e inequality. A measure $\mu$ is said to satisfy a Poincar\'e inequality with constant $\mathrm{C_p}(\mu)$, if for every smooth function $g:\RR^d\to \RR$,
$$\int\limits\limits_{\RR^d} g^2d\mu - \left(\int\limits_{\RR^d} gd\mu\right)^2 \leq \mathrm{C_p}(\mu)\int\limits_{\RR^d}\norm{\nabla g}_2^2d\mu,$$
where we implicitly assume that $\mathrm{C_p}(\mu)$ is the smallest constant for which this inequality holds. If $\mu$ satisfies such an inequality, then, in some sense, $\mu$ must be regular. Indeed, $\mu$ must have finite moments of all orders. For such measures we prove:
\begin{theorem} \label{thm: poincare stability}
	Let $\mu$ be a centered measure on $\RR^d$ with finite Poincar\'e constant $\mathrm{C_p}(\mu) < \infty$. Then
	$$\delta_{\mathrm{Tal}}\left(\mu\right) \geq \min\left(\frac{1}{4},\frac{(\mathrm{C_p}(\mu)+1)\left(2-2\mathrm{C_p}(\mu) +(\mathrm{C_p}(\mu)+1)\ln\left(\mathrm{C_p}(\mu)\right)\right)}{(\mathrm{C_p}(\mu)-1)^3}\right)\mathrm{D}(\mu||\gamma).$$
\end{theorem}
Note that as the deficit is invariant to translations, there is no loss in generality in assuming that $\mu$ is centered. Furthermore, the Poincar\'e constant tensorizes, in the sense that for any two measures $\nu$ and $\mu$,  $\mathrm{C_p}(\nu \otimes \mu) = \max\left(\mathrm{C_p}(\nu),\mathrm{C_p}(\mu)\right)$. So, if $\mu$ is a product measure $\mathrm{C_p}(\mu)$ does not depend on the dimension and we regard it as a dimensionless quantity. For a more applicable form of the result we may use the inequality
$$\min\left(\frac{1}{4},\frac{(x+1)(2-2x+(x+1)\ln(x))}{(x-1)^3}\right)\geq \frac{\ln(x+1)}{4x},$$
valid for $x > 0$, to get
$$\delta_{\mathrm{Tal}}\left(\mu\right) \geq \frac{\ln(\mathrm{C_p}(\mu)+1)}{4\mathrm{C_p}(\mu)}\mathrm{D}(\mu||\gamma).$$
Theorem \ref{thm: poincare stability} should be compared with Theorem 1 in \cite{fathi2016quantitative} and Theorem 7 in \cite{eldan2019shannon} which give similar stability estimates, involving the Poincar\'e constant, for the log-Sobolev and Shannon-Stam inequalities.\\

Regarding the conditions of the theorem; as will be shown in Section \ref{sec: instability} below, there exists a measure $\mu$ for which $\delta_{\mathrm{Tal}}(\mu)$ may be arbitrarily close to $0$, while $\mathcal{W}_2\left(\mu,\gamma\right)$ remains bounded away from $0$. Thus, in order to establish meaningful stability results, in relative entropy, it is necessary to make some assumptions on the measure $\mu$.\\

In case the measure $\mu$ does not satisfy a Poincar\'e inequality, we provide estimates in terms of its covariance matrix. It turns out, that if $\mathrm{Cov}(\mu)$ is strictly smaller than the identity, at least in some directions, we may still produce a dimension-free bound for $\delta_{\mathrm{Tal}}(\mu).$
\begin{theorem} \label{thm: small covariance}
	Let $\mu$ be a centered measure on $\RR^d$ and let $\{\lambda_i\}_{i=1}^d$ be the eigenvalues of $\mathrm{Cov}\left(\mu\right)$, counted with multiplicity. Then
	$$\delta_{\mathrm{Tal}}\left(\mu\right)\geq \sum\limits_{i=1}^d\frac{2(1 - \lambda_i) + (\lambda_i+1)\log(\lambda_i)}{\lambda_i - 1}\mathbbm{1}_{\{\lambda_i < 1\}}.$$
\end{theorem}
Remark that for $0< x < 1$, the function $g(x):=\frac{2(1 - x) + (x+1)\log(x)}{x - 1}$ is positive and that it is a decreasing function of $x$. Also, it can be verified that $g'$ is actually concave on this domain, from which we may see $g(x) \geq \frac{1}{6}(x-1)^2$. Thus, if $\mathrm{Cov}(\mu) \preceq \mathrm{I}_d$, then the Theorem implies the weaker result
$$\delta_{\mathrm{Tal}}(\mu) \geq \frac{1}{6}\norm{\mathrm{I}_d-\mathrm{Cov}(\mu)}_{HS}^2,$$
where $\norm{\cdot}_{HS}$, stands for the Hilbert-Schmidt norm.  In line with the above discussion, we may regard $\norm{\mathrm{Cov}(\mu) - \mathrm{I}_d}_{HS}^2$ as a certain distance between $\mu$ and the standard Gaussian. Theorem 3 in \cite{eldan2019stability} gives a similar estimate for the log-Sobolev inequality. Indeed, our methods are based on related ideas.\\

If $\mathrm{Cov}(\mu) = \mathrm{I}_d$, Theorem \ref{thm: small covariance} does not give any new insight beyond \eqref{eq: talagrand}. The next result applies, among others, to this case.
\begin{theorem} \label{thm: small trace}
	Let $\mu$ be a centered measure on $\RR^d$, such that $\mathrm{Tr}\left(\mathrm{Cov}(\mu)\right) \leq d$. Then
	$$\delta_{\mathrm{Tal}}(\mu) \geq \min\left(\frac{\mathrm{D}(\mu||\gamma)^2}{6d},\frac{\mathrm{D}(\mu||\gamma)}{4}\right).$$
\end{theorem} 
As opposed to the previous two results, Theorem \ref{thm: small trace} is not dimension-free and is directly comparable to \eqref{eq: fathi bound}. Under the assumption $\mathrm{Tr}\left(\mathrm{Cov}(\mu)\right) \leq d$, by using \eqref{eq: wass1} we may view the theorem as a strengthening of \eqref{eq: fathi bound}. We should also comment that by Pinsker's inequality (\cite{cover2012elements}), relative entropy induces a stronger topology than the $\mathcal{W}_1$ metric. On the other hand, \eqref{eq: fathi bound} holds in greater generality than Theorem \ref{thm: small trace} as it makes no assumptions on the measure $\mu$. It is then natural to ask whether one can relax the conditions of the theorem. We give a negative answer to this question.
\begin{theorem} \label{thm: instability}
	Fix $d \in \mathbb{N}$ and let $\xi > d$. There exist a sequence of centered measures $\mu_k$ on $\RR^d$ such that:
	\begin{itemize}
		\item $\lim\limits_{k \to \infty}\mathrm{Tr}\left(\mathrm{Cov}(\mu_k)\right) = \xi.$
		\item $\lim\limits_{k \to \infty}\delta_{\mathrm{Tal}}\left(\mu_k\right) = 0.$
		\item  $\liminf\limits_{k \to \infty}\mathcal{W}_2^2(\mu_k,\gamma) \geq  \xi-d > 0.$
	\end{itemize}
\end{theorem}
Thus, even for one dimensional measures, in order to obtain general stability estimates in relative entropy or even in the quadratic Wasserstein distance, the assumption $\mathrm{Tr}\left(\mathrm{Cov}(\mu)\right) \leq d$ is necessary.\\

The counterexample to stability, guaranteed by Theorem \ref{thm: instability}, may be realized as a Gaussian mixture. In fact, as demonstrated by recent works (\cite{eldan2019stability, courtade2018quantitative, carlen2008entropy}), Gaussian mixtures may serve as counterexamples to stability of several other Gaussian functional inequalities. This led the authors of \cite{eldan2019stability} to note that if a measure $\mu$ saturates the log-Sobolev inequality, then it must be close, in $L_2$-Wasserstein distance, to some Gaussian mixture. We show that this is also true, in relative entropy, for Talagrand's inequality.
\begin{theorem} \label{thm: mixture stability}
	Let $\mu$ be a centered measure on $\RR^d$. Then there exists another measure $\nu$ with $\mathrm{Cov}(\nu) \preceq \mathrm{Cov}(\mu)$, such that 
	if $\delta_{\mathrm{Tal}}(\mu) \geq d$,
	$$\delta_{\mathrm{Tal}}(\mu)\geq \frac{\mathrm{D}\left(\mu||\nu*\gamma\right)}{6},$$
	and if $\delta_{\mathrm{Tal}}(\mu) < d,$
	$$\delta_{\mathrm{Tal}}(\mu)\geq \frac{1}{3\sqrt{3}}\frac{\mathrm{D}(\mu||\nu*\gamma)^{\frac{3}{2}}}{\sqrt{d}}.$$
\end{theorem}
Note that, in light of Theorem \ref{thm: instability}, the above theorem is not true without the convolution, and we cannot, in general, replace $\nu * \gamma$ by $\gamma$.\\
For our last result, define the Fisher information of $\mu$, relative to $\gamma$, as
$$\mathrm{I}(\mu||\gamma):=\int\limits_{\RR^d}\norm{\nabla \ln\left(\frac{d\mu}{d\gamma}\right)}_2^2d\mu.$$
Gross' log-Sobolev inequality (\cite{gross1975logarithmic}) states that
$$\mathrm{I}(\mu||\gamma) \geq 2\mathrm{D}(\mu||\gamma).$$
For this we define the deficit as
$$\delta_{\mathrm{LS}}(\mu) = \mathrm{I}(\mu||\gamma) - 2\mathrm{D}(\mu||\gamma).$$
As will be described in Section \ref{sec: method} below, our approach draws a new connection between Talagrand's and the log-Sobolev inequalities.
One benefit of this approach is that all of our results apply verbatim to the log-Sobolev inequality. Some of the results improve upon existing estimates in the literature. We summarize those in the following corollary.

\begin{corollary} \label{cor: ls}
	Let $\mu$ be a centered measure on $\RR^d$. Then there exists a measure $\nu$ such that $\mathrm{Cov}(\nu)\preceq \mathrm{Cov}(\mu)$ and
	$$\delta_{\mathrm{LS}}(\mu) \geq \min\left(\frac{1}{3\sqrt{3}}\frac{\mathrm{D}\left(\mu||\nu*\gamma\right)^{\frac{3}{2}}}{\sqrt{d}}, \frac{\mathrm{D}\left(\mu||\nu*\gamma\right)}{6}\right).$$
	Moreover, if $\mathrm{Tr}\left(\mathrm{Cov}\left(\mu\right)\right)\leq d$ then
	$$\delta_{\mathrm{LS}}(\mu) \geq \min\left(\frac{\mathrm{D}(\mu||\gamma)^2}{6d},\frac{\mathrm{D}(\mu||\gamma)}{4}\right),$$
\end{corollary}
 
The second point of the corollary is an improvement of Corollary 1.2 in \cite{bobkov2014bounds} which shows, under the same hypothesis,
$$\delta_{\mathrm{LS}}(\mu) \geq c\frac{\mathcal{W}_2^4(\mu,\gamma)}{d},$$
for some universal constant $c >0$. The improved bound can actually be deduced from Theorem 1.1 in the same paper, but it does not seem to appear in the literature explicitly\\
The first point of Corollary \ref{cor: ls} strengthens Theorem 7 in \cite{eldan2019stability} which states, that for some measure $\nu$:
\begin{equation} \label{eq: mixture stab}
\delta_{\mathrm{LS}}(\mu) \geq \frac{1}{15}\frac{\mathcal{W}_2^3(\mu,\nu*\gamma)}{\sqrt{d}}.
\end{equation}
Our proof closely resembles theirs, but our analysis yields bounds in the stronger relative entropy distance. The authors of \cite{eldan2019stability} raise the natural question, whether the dependence on the dimension in \eqref{eq: mixture stab} can be completely discarded. The same question is also relevant to $\delta_{\mathrm{Tal}}(\mu)$. We do not know the answer to either of the questions, which seem related.
\subsection*{Organization}
The remainder of the paper is organized as follows: In Section \ref{sec: instability} we give a counter-example to stability of Talagrand's inequality, proving Theorem \ref{thm: instability}. Section \ref{sec: method} is devoted to explaining our method and proving some of its basic properties which will then be used in Section \ref{sec: stability} to prove the stability estimates. Finally, in Section \ref{sec: application} we give an application of our results to Gaussian concentration inequalities.
\subsection*{Acknowledgments}
We wish to thank Ronen Eldan, Max Fathi, Renan Gross, Emanuel Indrei and Yair Shenfeld for useful discussions and for their comments concerning a preliminary draft of this work. We are also grateful to the anonymous referee for carefully reading  this  paper and  providing  thoughtful  comments.
\section{A counterexample to stability} \label{sec: instability}
In this section we show that one cannot expect any general stability result to hold if \linebreak $\mathrm{Tr}\left(\mathrm{Cov}(\mu)\right) > d$. We present a one-dimensional example, which may be easily generalized to higher dimensions. The following notations will be used in this section:
\begin{itemize}
	\item For $\sigma^2>0$, $\gamma_{\sigma^2}$ denotes the law of the centered $1$-dimensional Gaussian with variance $\sigma^2$.
	\item Fix $\xi>1$ and $k \in \NN$, we set  $$\mu_k:=\left(1-\frac{1}{k}\right)\gamma_1 + \frac{1}{k}\gamma_{k(\xi-1)}.$$
\end{itemize} 
Recall now the Kantorovich dual formulation (see \cite{gentil2017analogy,villani2008optimal}, for example) of the $L_2$-Wasserstien distance. For $\nu$ and $\mu$ measures on $\RR$, we have
\begin{equation} \label{eq: kantorovich}
\mathcal{W}_2^2(\mu, \nu) =\sup\limits_{g}\left\{\int\limits_{\RR}  g(x) d\mu(x) - \int\limits_{\RR}(Qg)(x) d\nu(x)\right\},
\end{equation}
where the supremum runs over all measurable functions, and $Qg$ denotes the sup-convolution of $g$, namely
$$Qg(x) = \sup\limits_{y\in \RR}\{g(y) - (x-y)^2\}.$$
\begin{proof}[Proof of Theorem \ref{thm: instability}]
We first note that $\mathrm{Var}(\mu_k) \xrightarrow{k\to \infty} \xi > 1$.
Towards understanding $\delta_{\mathrm{Tal}}(\mu_k)$ we use the fact that
relative entropy is convex with respect to mixtures of measures (\cite{cover2012elements}), so 
\begin{equation} \label{eq: entropy mixture}
\mathrm{D}(\mu_k||\gamma) \leq \frac{1}{k} \mathrm{D}\left(\gamma_{k(\xi-1)}||\gamma\right) = \frac{1}{2k}\left(k(\xi-1) - 1 - \ln\left(k(\xi-1) \right)\right)\leq\frac{\xi-1}{2}.
\end{equation}
To control the Wasserstein distance, define the functions
$$g_k(x) = \begin{cases}
0& \text{ if } |x| < \frac{\sqrt{k}}{\ln(k)}\\
\left(1 - \frac{1}{\ln(k)}\right)x^2& \text{ otherwise}
\end{cases}
.$$
The main idea is that as $k$ increases, $Qg_k$ vanishes in an ever expanding region, while growing slowly outside of the region. Formally, for $0 \leq x \leq \frac{\sqrt{k}}{\ln(k)} - \frac{\sqrt{k\left(\ln(k) - 1\right)}}{\ln(k)^{\frac{3}{2}}},$ it holds that 
$$g_k\left(\frac{\sqrt{k}}{\ln(k)}\right) - \left(x - \frac{\sqrt{k}}{\ln(k)}\right)^2 = \left(1 - \frac{1}{\ln(k)}\right)\left(\frac{\sqrt{k}}{\ln(k)}\right)^2 - \left(x-\frac{\sqrt{k}}{\ln(k)}\right)^2\leq 0.$$
and in particular, if $\frac{\sqrt{k}}{\ln(k)} < y$,
$$g_k(y) - \left(x - y\right)^2 < 0,$$
which shows $Qg_k(x) = 0$.
There exists a constant $c > 0$ such that 
$$\frac{\sqrt{k}}{\ln(k)} - \frac{\sqrt{k\left(\ln(k) - 1\right)}}{\ln(k)^{\frac{3}{2}}} \geq ck^{\frac{1}{4}},$$
which, combined with the previous observation shows that for $|x| \leq ck^{\frac{1}{4}}$, $Qg_k(x) = 0$. If $ |x| > ck^{\frac{1}{4}}$ it is standard to show $Qg_k(x) \leq \ln(k)x^2.$ 
So,
$$\int\limits_{\RR}Qg_k(x)d\gamma_1(x) \leq \ln(k)\int\limits_{|x| \geq ck^{\frac{1}{4}}}x^2d\gamma_1(x) = \ln(k)\left(\frac{c\sqrt{2}}{\sqrt{\pi}}k^{\frac{1}{4}}e^{-\frac{c^2\sqrt{k}}{2}}+\int\limits_{|x| \geq ck^{\frac{1}{4}}}d\gamma_1\right)\xrightarrow{k\to \infty} 0,$$
where the equality is integration by parts. Also, it is clear that
$$\int\limits_{\RR}g_k(x)d\gamma_1(x) \xrightarrow{k\to \infty} 0.$$
Now, if $\vphi$ denotes the density of the standard Gaussian, then by a change of variables we have
\begin{align*}
\frac{1}{k}\int\limits_{\RR}g_k(x)d\gamma_{k(\xi-1)}(x) &= \left(1-\frac{1}{\ln(k)}\right)\frac{1}{k}\int\limits_{|x| \geq \frac{\sqrt{k}}{\ln(k)} }\frac{x^2}{\sqrt{k\left(\xi-1\right)}}\vphi\left(\frac{x}{\sqrt{k\left(\xi-1\right)}}\right)dx\\
&=\left(1-\frac{1}{\ln(k)}\right)(\xi-1)\int\limits_{|y| \geq \frac{1}{\ln(k)\sqrt{\xi - 1}} }y^2\vphi(y)dy \xrightarrow{k\to \infty} \xi-1.\\
\end{align*}
Combining the above displays with \eqref{eq: kantorovich} we get,
	\begin{align*}
	\mathcal{W}_2^2(\mu_k,\gamma_1) &\geq\int\limits_{\RR} g_k(x) d\mu_k(x) - \int\limits_{\RR}Qg_k(x) d\gamma_1(x) \\
	&=  \left(1-\frac{1}{k}\right)\int\limits_{\RR}g_k(x) d\gamma_1(x) + \frac{1}{k}\int\limits_{\RR}g_k(x) d\gamma_{k(\xi-1)}(x) - \int\limits_{\RR}Qg_k(x) d\gamma_{1}(x)\xrightarrow{k\to \infty} \xi-1.\\
	\end{align*}
Finally, from \eqref{eq: entropy mixture} we obtain
\begin{align*}
\delta_{\mathrm{Tal}}(\mu_k) &= 2\mathrm{D}(\mu_k||\gamma_1) - \mathcal{W}_2^2\left(\mu_k,\gamma_1\right)\xrightarrow{k\to \infty} 0.
\end{align*}
\end{proof}
We remark that, in light of Theorem \ref{thm: mixture stability}, it would seem more natural to have as a counterexample a mixture of Gaussians with unit variance, as was done in \cite{eldan2019stability} for the log-Sobolev inequality. However, \eqref{eq: fathi bound} tells us that the situation in Talagrand's inequality is a bit more delicate, since the inequality is stable with respect to the $\mathcal{W}_1$ metric. Thus, as in the given example, a counterexample to stability (in the $\mathcal{W}_2$ metric or relative entropy) should satisfy $\lim\limits_{k \to \infty}\mathcal{W}_1(\mu_k,\gamma) = 0$, while $\liminf\limits_{k \to \infty}\mathcal{W}_2(\mu_k,\gamma) > 0$. keeping this in mind, it seems more straightforward to allow the second moments of the summands in the mixture to vary while keeping their means fixed at the origin. This is also very similar to the counterexample, obtained in \cite{courtade2018quantitative}, for the entropy power inequality.

\section{The F\"ollmer process} \label{sec: method} 
Our method is based on an entropy minimizing process, known in the literature as the F\"ollmer process. The high-level idea underlying this work is to use the process in order to embed a given measure as the terminal point of some martingale, in the Wiener space. This will induce a coupling between the measure and $\gamma$. As will be shown, the process also solves a variational problem, which turns out to yield a representation formula for the relative entropy. Combining these two properties will allow us to bound $\delta_{\mathrm{Tal}}(\mu)$ from below.\\
The process appears in the works of F\"ollmer (\cite{follmer1985entropy,follmer1986time}). It was later used by Borell in \cite{borell2003ehrhard} and Lehec in \cite{lehec2013representation} to give simple proofs of various functional inequalities, including Talagrand's Gaussian transport-entropy inequality.
Recently, the process was used in order to prove stability estimates for the Shannon-Stam (\cite{eldan2019shannon}) and log-Sobolev (\cite{eldan2019stability}) inequalities. In this section, we present the relevant details concerning the process. The reader is referred to \cite{lehec2013representation, eldan2018regularization,eldan2018clt} for further details and a more rigorous treatment. We will sketch the main ideas here for completeness. \\

Throughout this section we fix a measure $\mu$ on $\RR^d$ with expectation $0$, a finite second moment matrix and a density $f$, relative to $\gamma$. Consider the Wiener space $C([0,1],\RR^d)$ of continuous paths with the Borel sigma-algebra generated by the supremum norm $\norm{\cdot}_\infty$. We endow $C([0,1],\RR^d)$ with a probability measure $P$ and a process $B_t$ which is a Brownian motion under $P$. We will denote by $\omega$ elements of $C([0,1],\RR^d)$ and by $\FF_t$ the natural filtration of $B_t$. Define the measure $Q$ by
$$\frac{dQ}{dP}(\omega) = f(\omega_1).$$
$Q$ is absolutely continuous with respect to $P$, in which case, a converse to Girsanov's theorem implies that there exists a drift, $v_t$, adapted to $\FF_t$, in the Wiener space, such that the process 
\begin{equation} \label{eq: follmer process}
X_t := B_t + \int\limits_{0}^tv_s(X_s)ds,
\end{equation}
has the same law as $Q$, and that, under $Q$, $X_t$ is a Brownian motion. In particular, by construction, $X_1 \sim \mu$ and conditioned on $X_1$, $X_t$ serves a Gaussian bridge between $0$ and $X_1$. Thus, by the representation formula for Brownian bridges
\begin{equation} \label{eq: brownian bridge}
X_t \stackrel{\text{law}}{=} tX_1 + \sqrt{t(1-t)}G,
\end{equation}
where $G$ is a standard Gaussian, independent from $X_1$. We call $v_t(X_t)$ the F\"ollmer drift and $X_t$ the F\"ollmer process.
As $\mu$ and $\gamma$ are the laws of $X_1$ and $B_1$, it is now immediate that
\begin{align} \label{eq: entropy ineq}
\mathrm{D}(Q||P) \geq \mathrm{D}(\mu||\gamma).
\end{align}
A remarkable feature is that, since $\frac{dQ}{dP}$ depends only on the terminal points, the above is actually an equality and $\mathrm{D}(Q||P) = \mathrm{D}(\mu||\gamma)$. This implies that the drift, $v_t$, is a martingale (see Lemmas 10 and 11 in \cite{lehec2013representation}). \\
We now use Girsanov's theorem (\cite[Theorem 8.6.3]{oksendal2003stochastic}) to rewrite $\frac{dQ}{dP}$ as an exponential martingale,
$$\frac{dQ}{dP}(\omega) = \exp\left(-\int\limits_0^1v_t(\omega)dX_t(\omega) + \frac{1}{2}\int\limits_0^1\norm{v_t(\omega)}_2^2dt\right).$$
Under $Q$, $X_t$ is a Brownian motion, so
$$\mathrm D\left(Q||P\right) = \int\limits_{C([0,1],\RR^d)}\ln\left(\frac{dQ}{dP}\right)dQ = \frac{1}{2}\int\limits_{0}^1\EE\left[\norm{v_t(X_t)}_2^2\right]dt,$$
which gives the formula
\begin{equation} \label{eq: follmer energy}
\mathrm{D}(\mu||\gamma) = \frac{1}{2}\int\limits_{0}^1\EE\left[\norm{v_t(X_t)}_2^2\right]dt.
\end{equation}
For simplicity, from now on, we suppress the dependence of $v_t$ on $X_t$. Combining the above with \eqref{eq: entropy ineq} shows that among all adapted drifts $u_t$ such that $\mu \sim B_1 + \int\limits_0^1u_tdt$, $v_t$ minimizes the energy in the following sense
\begin{equation} \label{eq: variatinal follmer}
v_t = \arg\min\limits_{u_t} \frac{1}{2}\int\limits_{0}^1\EE\left[\norm{u_t}_2^2\right]dt.
\end{equation}  
Theorem 12 in \cite{lehec2013representation} capitalizes on the structure of $\frac{dP}{dQ}$ to give an explicit representation of $v_t$ as
\begin{equation}\label{eq: explicitvt}
v_t = \nabla\ln\left(P_{1-t}f(X_t)\right).
\end{equation}
where $P_{1-t}$ denotes the heat semi-group.
Since $v_t$ is a martingale, It\^o's formula shows
\begin{equation*} 
dv_t = \nabla v_tdB_t =  \nabla^2\ln\left(P_{1-t}f(X_t)\right)dB_t.
\end{equation*}

Lehec's proof of Talagrand's transport-entropy inequality relied on the fact that \eqref{eq: follmer process} induces a natural coupling between $\mu$ and $\gamma$ so that, by Jensen's inequality
$$\mathcal{W}_2^2\left(\mu,\gamma\right) \leq \EE\left[\norm{X_1 - B_1}_2^2\right] \leq \int\limits_0^1\EE\left[\norm{v_t}_2^2\right]dt = 2\mathrm{D}(\mu||\gamma).$$
Our goal is to make this quantitative.
\subsection{The martingale approach}
As was demonstrated in \cite{eldan2019shannon} and \cite{eldan2018clt} it is often easier to work with an equivalent martingale formulation of the F\"ollmer drift. Consider the Doob martingale $\EE\left[X_1|\FF_t\right]$. By the martingale representation theorem (\cite[Theorem 4.33]{oksendal2003stochastic}) there exists a uniquely-defined, adapted, matrix valued process $\Gamma_t$ which satisfies
\begin{equation} \label{eq: gammadef}
\EE\left[X_1|\FF_t\right] = \int\limits_0^t\Gamma_sdB_s.
\end{equation}
We claim that
\begin{equation} \label{eq: gamma follmer}
v_t = \int\limits_{0}^t\frac{\Gamma_s-\mathrm{I}_d}{1-s}dB_s.
\end{equation}
Indeed, by Fubini's theorem
\begin{align*}
\int\limits_0^1\Gamma_sdB_s = \int\limits_0^1\mathrm{I}_ddB_s + \int\limits_0^1\left(\Gamma_s-\mathrm{I}_d\right)dB_s &= B_1 + \int\limits_0^1\int\limits_s^1\frac{\Gamma_s-\mathrm{I}_d}{1-s}dtdB_s\\
&=B_1 + \int\limits_0^1\int\limits_0^t\frac{\Gamma_s-\mathrm{I}_d}{1-s}dB_sdt.
\end{align*}
For the moment denote $\tilde{v}_t := \int\limits_0^t\frac{\Gamma_s-\mathrm{I}_d}{1-s}dB_s$. Since $v_t$ is a martingale $v_t - \tilde{v}_t$ is a martingale as well and the above shows that for every $t \in [0,1]$, almost surely,
$$\int\limits_t^1(v_s - \tilde{v}_s)ds|\FF_t = 0.$$
This implies the identity \eqref{eq: gamma follmer}.
In particular, from \eqref{eq: explicitvt}, $\Gamma_t$ turns out to be symmetric, which shows, using It\^o's formula,
\begin{equation} \label{eq: entropy}
2\mathrm{D}(\mu||\gamma) = \int\limits_0^1\EE\left[\norm{v_t}_2^2\right]dt =\mathrm{Tr}\int\limits_0^1\int\limits_0^t\frac{\EE\left(\Gamma_s-\mathrm{I}_d\right)^2}{(1-s)^2}dsdt= \mathrm{Tr}\int\limits_0^1\frac{\EE\left(\Gamma_t-\mathrm{I}_d\right)^2}{1-t}dt.
\end{equation}
Also, note that $$B_1 + \int\limits_0^1\left(\Gamma_t - \mathrm{I}_d\right)dB_t = \int\limits_0^1\Gamma_tdB_t \sim \mu,$$
which implies
\begin{equation} \label{eq: wasserstein}
\mathcal{W}_2^2(\mu,\gamma)\leq \mathrm{Tr}\int\limits_0^1\EE\left[\left(\Gamma_t - \mathrm{I}_d\right)^2\right]dt.
\end{equation}
As $X_1 \sim \mu$, from \eqref{eq: explicitvt} we get
\begin{equation}\label{eq: fisher}
\mathrm{Tr}\int\limits_0^1\frac{\EE\left[\left(\Gamma_t- \mathrm{I}_d\right)^2\right]}{(1-t)^2}dt = \EE\left[\norm{v_1}_2^2\right] = \int\limits_{\RR^d}\norm{\nabla \ln(f(x))}_2^2d\mu(x) = \mathrm{I}(\mu||\gamma).
\end{equation}
Combining \eqref{eq: entropy},\eqref{eq: wasserstein},\eqref{eq: fisher}, we see a very satisfying connection between the log-Sobolev and Talagrand's transport-entropy inequalities, as 
$$\mathrm{Tr}\int\limits_0^1\frac{\EE\left[\left(\Gamma_t- \mathrm{I}_d\right)^2\right]}{(1-t)^2}dt \geq \mathrm{Tr}\int\limits_0^1\frac{\EE\left[\left(\Gamma_t- \mathrm{I}_d\right)^2\right]}{1-t}dt \geq \mathrm{Tr}\int\limits_0^1\EE\left[\left(\Gamma_t- \mathrm{I}_d\right)^2\right]dt,$$
implies
$$\mathrm{I}(\mu||\gamma) \geq 2\mathrm{D}(\mu||\gamma) \geq \mathcal{W}_2^2\left(\mu,\gamma\right).$$
In addition to its elegance, this representation can prove useful in the study of stability properties for those functional inequalities. 
We have the following representation for the deficits,
\begin{equation} \label{eq: deltals}
\delta_{\mathrm{LS}}\left(\mu\right) = \mathrm{Tr}\int\limits_0^1t\cdot\frac{\EE\left[\left(\Gamma_t- \mathrm{I}_d\right)^2\right]}{(1-t)^2}dt,
\end{equation}
\begin{equation} \label{eq: deltatal}
\delta_{\mathrm{Tal}}\left(\mu\right) \geq \mathrm{Tr}\int\limits_0^1t\cdot\frac{\EE\left[\left(\Gamma_t- \mathrm{I}_d\right)^2\right]}{1-t}dt.
\end{equation}
The above formulas are the key to Corollary \ref{cor: ls}.
\begin{proof}[Proof of Corollary \ref{cor: ls}]
	Note that by \eqref{eq: deltals} and \eqref{eq: deltatal}, any estimate on $\delta_{\mathrm{Tal}}(\mu)$ which is achieved by bounding $$\mathrm{Tr}\int\limits_0^1t\cdot\frac{\EE\left[\left(\Gamma_t- \mathrm{I}_d\right)^2\right]}{1-t}dt$$
	 from below will also imply a bound for $\delta_{\mathrm{LS}}(\mu)$. Since Theorem \ref{thm: small trace} and Theorem \ref{thm: mixture stability} are proved using this method, the corollary follows.
\end{proof}
One should remark that, by using \eqref{eq: explicitvt} and \eqref{eq: gamma follmer}, we have the identification
$$\nabla^2 \ln\left(P_{1-t} f(X_t)\right) = \frac{\Gamma_t - \mathrm{I}_d}{1-t},$$
Thus, all of the calculations above could have been done with respect to $v_t$, without appealing to the Doob martingale in \eqref{eq: gammadef}. \\

However, we hope that by introducing the matrix $\Gamma_t$, we may shed some new light on the subject, which may be of independent use elsewhere. As an example, consider the bounds proved in \cite[Theorem 6]{eldan2018clt}, for the rate of convergence in the entropic central limit theorem and in \cite[Lemma 2]{eldan2019stability} for stability of the Shannon-Stam inequality. In all cases, the proof of the quantitative estimates boils down to showing that the matrix $\Gamma_t$ is different than the identity, far from $0$. \\

We also feel that some calculations become more natural when dealing with the matrix $\Gamma_t$. As an easy example for the use of the above formulation, consider the case in which $\delta_{\mathrm{Tal}}(\mu) = 0$. By \eqref{eq: deltatal}, it follows that, for every $t \in[0,1]$, $\Gamma_t = \mathrm{I}_d$ almost surely. Thus, since $\mu \sim \int\limits_0^1\Gamma_tdB_t$, $\mu$ must be the standard Gaussian, which is known to be the only centered equality case (Of-course one could also conclude that $f \equiv 1$, which yields the same result). 

\subsection{Properties of the F\"ollmer process}
Our objective is now clear: In order to produce any stability estimates it will be enough to show, roughly speaking, that the process $\Gamma_t$ is far from $\mathrm{I}_d$, not too close to time $0$.  
In order to establish such claims we will use several other properties of the processes $\Gamma_t,v_t$, which we now state and prove.
First, as in \cite[Lemma 11]{eldan2018clt} it is possible use \eqref{eq: gamma follmer} along with integration by parts to obtain the identity:
\begin{align} \label{eq: int by parts}
\EE\left[v_t\otimes v_t\right] = \frac{\EE\left[\mathrm{I}_d-\Gamma_t\right]}{1-t} + \left(\mathrm{Cov}(\mu) - \mathrm{I}_d\right).
\end{align}
Combining the fact that $v_t$ is a martingale with \eqref{eq: gamma follmer} we also see
\begin{align} \label{eq: general bound}
\frac{d}{dt}\EE\left[\norm{v_t}_2^2\right] &=\mathrm{Tr}\frac{\EE\left[\left(\mathrm{I}_d-\Gamma_t\right)^2\right]}{(1-t)^2} \nonumber\\
&\geq \frac{1}{d}\left(\mathrm{Tr}\left(\frac{\EE\left[\mathrm{I}_d-\Gamma_t\right]}{1-t}\right)\right)^2 = \frac{\left(\EE\left[\norm{v_t}_2^2\right] - \mathrm{Tr}\left(\mathrm{Cov}(\mu)-\mathrm{I}_d\right)\right)^2}{d},
\end{align}
where we have used Cauchy-Schwartz for the inequality. Using this we prove the following two lemmas:
\begin{lemma} \label {lem: dgammat}
	It holds that
	$$\frac{d}{dt}\EE\left[\Gamma_t\right] = \frac{\EE\left[\Gamma_t\right] -\EE\left[\Gamma_t^2\right]}{1-t}.$$
\end{lemma}
\begin{proof}
    Since $\Gamma_t$ is a symmetric matrix equation \eqref{eq: gamma follmer} implies
    $$\frac{d}{dt}\EE\left[v_t\otimes v_t\right] = \frac{\EE\left[(\mathrm{I}_d-\Gamma_t)^2\right]}{(1-t)^2}.$$
    Combined with \eqref{eq: int by parts}, this gives
	$$\frac{\EE\left[(\mathrm{I}_d-\Gamma_t)^2\right]}{(1-t)^2} = \frac{d}{dt}\frac{\EE\left[\mathrm{I}_d-\Gamma_t\right]}{1-t} = \frac{\EE\left[\mathrm{I}_d - \Gamma_t\right]- (1-t)\frac{d}{dt}\EE\left[\Gamma_t\right]}{(1-t)^2}.$$
	Rearranging the terms yields the result.
\end{proof}
\begin{lemma} \label{lem: dimension comparision}
	Suppose that $\mathrm{Tr}\left(\mathrm{Cov}(\mu)\right) \leq d$ and let $v_t$ be as defined above. Then:
	\begin{itemize}
		\item For $0 \leq t \leq \frac{1}{2}$, $\EE\left[\norm{v_t}_2^2\right] \leq \EE\left[\norm{v_{1/2}}_2^2\right]\frac{2d}{\EE\left[\norm{v_{1/2}}_2^2\right]\left(1-2t\right) + 2d}$.
		\item For $\frac{1}{2} \leq t \leq 1$, $\EE\left[\norm{v_t}_2^2\right] \geq \EE\left[\norm{v_{1/2}}_2^2\right]\frac{2d}{\EE\left[\norm{v_{1/2}}_2^2\right]\left(1-2t\right) + 2d}$.
	\end{itemize}
\end{lemma}
\begin{proof}
	Since $\mathrm{Tr}\left(\mathrm{Cov}(\mu)\right) \leq d$, \eqref{eq: general bound} gives
	$$\frac{d}{dt}\EE\left[\norm{v_t}_2^2\right]\geq \frac{\left(\EE\left[\norm{v_t}_2^2\right]\right)^2}{d}.$$
	The unique solution to the differential equation 
	$$g'(t) = \frac{g(t)^2}{d}, \text{ with inital condition } g\left(\frac{1}{2}\right)=\EE\left[\norm{v_{1/2}}_2^2\right],$$
	is given by 
	$$g(t) = \EE\left[\norm{v_{1/2}}_2^2\right]\frac{2d}{\EE\left[\norm{v_{1/2}}_2^2\right]\left(1-2t\right) + 2d}.$$
	The result follows by Gronwall's inequality
\end{proof}
To get a different type of inequality, but of similar flavor, recall \eqref{eq: brownian bridge}, 
 $$X_t \stackrel{\text{law}}{=} tX_1 + \sqrt{t(1-t)}G,$$ where $G$ is a standard Gaussian, independent from $X_1$. Now, suppose that $\mu$ satisfies a Poincar\'e inequality with optimal constant $\mathrm{C_p}(\mu)$. In this case $X_t$ satisfies a Poincar\'e inequality with a constant, smaller than, $t^2\mathrm{C_p}(\mu) + t(1-t)$. This follows from the fact that the Poincar\'e constant is sub-additive with respect to convolutions (\cite{borovkov1984inequality}) and that if $X\sim \nu$ and $aX \sim \nu_a$ for some $a \in \RR$, then $\mathrm{C_p}(\nu_a ) = a^2\mathrm{C_p}(\nu).$ Applying the Poincar\'e inequality to $v_t(X_t)$, we get
\begin{equation} \label{eq: follmer poincare}
\EE\left[\norm{v_t}_2^2\right] \leq  \left(t^2\mathrm{C_p}(\mu) + t(1-t)\right)\left[\norm{\nabla v_t}_2^2\right]= \left(t^2\mathrm{C_p}(\mu) + t(1-t)\right)\frac{d}{dt}\left[\norm{v_t}_2^2\right],
\end{equation}
where the equality is due to the fact that $v_t$ is a martingale.
Repeating the proof of Lemma \ref{lem: dimension comparision} for the differential equation
$$g(t) = \left(t^2\mathrm{C_p}(\mu) + t(1-t)\right)g'(t), \text{ with inital condition } g\left(\frac{1}{2}\right)=\EE\left[\norm{v_{1/2}}_2^2\right],$$
proves:
\begin{lemma} \label{lem: comparision}
	Assume that $\mu$ has a finite Poincar\'e constant $\mathrm{C_p}(\mu) < \infty$. Then, for $v_t$ defined as above:
	\begin{itemize}
		\item For $0 \leq t \leq \frac{1}{2},$
		$$\EE\left[\norm{v_t}_2^2\right] \leq  \EE\left[\norm{v_{1/2}}_2^2\right]\frac{\left(\mathrm{C_p}(\mu)+1\right)t}{\left(\mathrm{C_p}(\mu)-1\right)t + 1}.$$
		\item For $\frac{1}{2}\leq t \leq 1,$
		$$\EE\left[\norm{v_t}_2^2\right] \geq  \EE\left[\norm{v_{1/2}}_2^2\right]\frac{\left(\mathrm{C_p}(\mu)+1\right)t}{\left(\mathrm{C_p}(\mu)-1\right)t + 1}.$$
	\end{itemize} 
\end{lemma}
\section{Stability for Talagrand's transportation-entropy inequality} \label{sec: stability}
We begin this section by showing two ways the F\"ollmer process may be used to establish quantitative stability estimates.
As before, $\mu$ is a fixed measure on $\RR^d$ with finite second moment matrix. $\Gamma_t$ and $v_t$ are defined as in the previous section.
Fix $t_0 \in[0,1]$, by \eqref{eq: deltatal}, we see
$$\delta_{\mathrm{Tal}}(\mu) \geq t_0 \mathrm{Tr}\int\limits_{t_0}^1\frac{\EE\left[\left(\mathrm{I}_d-\Gamma_t\right)^2\right]}{1-t}dt.$$
Now, using \eqref{eq: gamma follmer}, we obtain, by Fubini's theorem,
$$\int\limits_{t_0}^1\left(\EE\left[\norm{v_s}^2_2\right]- \EE\left[\norm{v_{t_0}}_2^2\right]\right)ds = \mathrm{Tr}\int\limits_{t_0}^1\int\limits_{t_0}^s\frac{\EE\left[\left(\mathrm{I}_d - \Gamma_t\right)^2\right]}{(1-t)^2}dtds = \mathrm{Tr}\int\limits_{t_0}^1\frac{\EE\left[\left(\mathrm{I}_d - \Gamma_t\right)^2\right]}{1-t}dt, $$
and
\begin{equation} \label{eq: truncation}
\delta_{\mathrm{Tal}}(\mu) \geq t_0\left(\int\limits_{t_0}^1\EE\left[\norm{v_t}_2^2\right]dt - (1-t_0)\EE\left[\norm{v_{t_0}}_2^2\right]\right)\geq t_0(1-t_0)\left(2\mathrm{D}(\mu||\gamma)-\EE\left[\norm{v_{t_0}}_2^2\right]\right),
\end{equation}
where we have used \eqref{eq: entropy} and the fact that $v_t$ is a martingale.
Another useful bound will follow by applying \eqref{eq: gamma follmer} to rewrite \eqref{eq: deltatal} as
$$\delta_{\mathrm{Tal}}(\mu) \geq \mathrm{Tr}\int\limits_0^1t(1-t)\cdot\frac{\EE\left[\left(\Gamma_t- \mathrm{I}_d\right)^2\right]}{(1-t)^2}dt = \int\limits_0^1t(1-t)\frac{d}{dt}\EE\left[\norm{v_t}_2^2\right]dt.$$
Integration by parts then gives
\begin{equation}\label{eq: halving}
\delta_{\mathrm{Tal}}(\mu) \geq \int\limits_0^1(2t-1)\EE\left[\norm{v_t}_2^2\right]dt.
\end{equation}
At an informal level, the above formula becomes useful if one is able to show that $\EE\left[\norm{v_t}_2^2\right]$ is large for $t \geq \frac{1}{2}$ and small otherwise.
\subsection{Measures with a finite Poincar\'e constant}
We now assume that the measure $\mu$ has a finite Poincar\'e constant $\mathrm{C_p}(\mu) < \infty$. 
\begin{proof} [Proof of Theorem \ref{thm: poincare stability}]
	First, suppose that $\EE\left[\norm{v_{1/2}}_2^2\right] \leq \mathrm{D}(\mu||\gamma)$. In this case \eqref{eq: truncation} shows
	$$\delta_{\mathrm{Tal}}\geq \frac{1}{4}\mathrm{D}(\mu||\gamma).$$
	Otherwise, $\EE\left[\norm{v_{1/2}}_2^2\right] > \mathrm{D}(\mu||\gamma)$, and plugging Lemma \ref{lem: comparision} into \eqref{eq: halving} shows
	\begin{align*} 
	\delta_{\mathrm{Tal}}(\mu) &\geq \mathrm{D}(\mu||\gamma)\int\limits_0^1(2t-1)\frac{(\mathrm{C_p}(\mu) + 1)t}{(\mathrm{C_p}(\mu) - 1)t +1}dt\nonumber\\
	&=\mathrm{D}(\mu||\gamma)\frac{(\mathrm{C_p}(\mu)+1)\left(2-2\mathrm{C_p}(\mu) +(\mathrm{C_p}(\mu)+1)\ln\left(\mathrm{C_p}(\mu)\right)\right)}{(\mathrm{C_p}(\mu)-1)^3},
	\end{align*}
	where the equality relies on the fact
	\begin{align*}
\frac{d}{dt}&\frac{(\mathrm{C_p}(\mu)+1)\left((\mathrm{C_p}(\mu)-1)t(\mathrm{C_p}(\mu)(t-1) -1-t)+(\mathrm{C_p}(\mu)+1)\ln\left((\mathrm{C_p}(\mu)-1)t+1\right)\right)}{\left(\mathrm{C_p}(\mu) - 1\right)^3}\\
	&= (2t-1)\frac{(\mathrm{C_p}(\mu)+1)t}{(\mathrm{C_p}(\mu) - 1)t +1}.
	\end{align*}
	The proof is complete.
\end{proof}
\subsection{Measures with small covariance}
Here we work under the assumption  $\mathrm{Tr}(\mathrm{Cov}(\mu)) \leq d$ and prove Theorem \ref{thm: small trace}.
\begin{proof}[Proof of Theorem \ref{thm: small trace}]
	Denote $c_\mu = \EE\left[\norm{v_{1/2}}_2^2\right]$. We begin by considering the case $c_\mu \leq \mathrm{D}(\mu||\gamma)$. In this case, \eqref{eq: truncation} shows
	$$\delta_{\mathrm{Tal}}(\mu)\geq \frac{1}{4}\mathrm{D}(\mu||\gamma).$$
	In the other case, $c_\mu > \mathrm{D}(\mu||\gamma)$ and Lemma \ref{lem: dimension comparision}, along with \eqref{eq: halving}, gives
	\begin{align*}
	\delta_{\mathrm{Tal}}(\mu) &\geq 2d\int\limits_0^1\frac{c_\mu(2t-1)}{c_\mu\left(1-2t\right) + 2d}dt\\
	&=2d\left(\frac{-d\ln\left(c_\mu+2d-2c_\mu t\right)-c_\mu t}{c_\mu}\right)\Big\vert_0^1 \\
	&=\frac{2d\left(d\ln(2d +c_\mu)-d\ln(2d-c_\mu)-c_\mu\right)}{c_\mu}\\ &=2d\left(\frac{2d\coth^{-1}\left(\frac{2d}{c_\mu}\right)}{c_\mu}-1\right).
	\end{align*}
	Note that \eqref{eq: int by parts} implies $c_\mu\leq 2d$, so the above is well defined. Also, for any $x \geq 1$, we have the inequality $\coth^{-1}(x)\cdot x - 1\geq \frac{1}{3x^2}$, applying it to the previous bound then gives
	$$\delta_{\mathrm{Tal}}(\mu)\geq \frac{c_\mu^2}{6d}>\frac{\mathrm{D}\left(\mu||\gamma\right)^2}{6d}.$$
\end{proof}
We can get a dimension free bound by considering directions $v \in \RR^d$ in which $\mathrm{Cov}(\mu)$ is strictly smaller than the identity. For this we use Lemma \ref{lem: dgammat} to establish:
$$\frac{d}{dt}\EE\left[\Gamma_t\right] = \frac{\EE\left[\Gamma_t\right] - \EE\left[\Gamma_t^2\right]}{1-t} \preceq \frac{\EE\left[\Gamma_t\right] - \EE\left[\Gamma_t\right]^2}{1-t}.$$
Fix $v \in \RR^d$, a unit vector, and define $f(t)=\left\langle v, \EE\left[\Gamma_t\right]v\right\rangle$. As $\EE\left[\Gamma_t\right]$ is symmetric, by Cauchy-Schwartz
$$\left\langle v,\EE\left[\Gamma_t\right]v\right\rangle^2 \leq \left\langle v,\EE\left[\Gamma_t\right]^2v\right\rangle.$$
This implies
$$\frac{d}{dt}f(t) \leq  \frac{f(t)(1-f(t))}{1-t}.$$
If $\left\langle v, \EE\left[\Gamma_0\right]v\right\rangle = \lambda$, from Gronwall's inequality we get
\begin{equation} \label{eq: gamma0}
\left\langle v, \EE\left[\Gamma_t\right]v\right\rangle\leq \frac{\lambda}{(\lambda-1)t + 1}.
\end{equation}
Using this, we prove Theorem \ref{thm: small covariance}.
\begin{proof}[Proof of Theorem \ref{thm: small covariance}]
	 For $\lambda_i < 1$, let $w_i$ be the unit eigenvector of $\mathrm{Cov}(\mu)$, corresponding to $\lambda_i$. From \eqref{eq: gamma0} we deduce, for every $t \in [0,1]$,
	 $$0\leq\left\langle w_i, \EE\left[\Gamma_t\right]w_i\right\rangle \leq 1.$$
	 We now observe that as $v_t$ is a martingale, and since $\mu$ is centered, it must hold that $v_0 = 0$, almost surely. 
	 Combining this with \eqref{eq: int by parts} shows $\EE\left[\Gamma_0\right] = \mathrm{Cov}(\mu)$ and in particular
	 $$\left\langle w_i, \EE\left[\Gamma_0\right]w_i\right\rangle = \lambda_i.$$
	 Using \eqref{eq: gamma0} and the fact that $\EE\left[\Gamma_t\right]$ is symmetric, we obtain:
	\begin{align*}
	 t\frac{\left\langle w_i,\EE\left[\left(\mathrm{I}_d-\Gamma_t\right)^2\right]w_i\right\rangle}{1-t} \geq  t\frac{\left(\left\langle w_i,\EE\left[\mathrm{I}_d-\Gamma_t\right]w_i\right\rangle\right)^2}{1-t} \geq \frac{t\left(1 - \frac{\lambda_i}{\left(\lambda_i-1\right)t + 1}\right)^2}{1-t} = t(1-t)\left(\frac{\lambda_i-1}{(\lambda_i-1)t+1}\right)^2.
	\end{align*}
	 So, by \eqref{eq: deltatal},
	 \begin{align*}
	 \delta_{\mathrm{Tal}}\left(\mu\right)&\geq \mathrm{Tr}\int\limits_0^1t\cdot\frac{\EE\left[\left(\mathrm{I}_d - \Gamma_t\right)^2\right]}{1-t}dt \geq \sum_{i=1}^d\mathbbm{1}_{\{\lambda_i < 1\}}\int\limits_0^1t\cdot\frac{\left\langle v_i,\EE\left[\left(\mathrm{I}_d - \Gamma_t\right)^2\right]v_i\right\rangle}{1-t}dt\\
	 &\geq \sum_{i=1}^d\mathbbm{1}_{\{\lambda_i < 1\}}\int\limits_0^1t(1-t)\left(\frac{\lambda_i-1}{(\lambda_i-1)t+1}\right)^2dt \\
	 &= \sum_{i=1}^d\frac{2(1 - \lambda_i) + (\lambda_i+1)\log(\lambda_i)}{\lambda_i - 1}\mathbbm{1}_{\{\lambda_i < 1\}}.
	 \end{align*}
\end{proof}
\subsection{Stability with respect to Gaussian mixtures}
In this section we prove Theorem \ref{thm: mixture stability}. Our proof is based on \cite{eldan2019stability}, but we use our framework to give an improved analysis. To control the relative entropy we will use a specialized case of the bound given in \cite{eldan2018clt}. We supply here a sketch of the proof for completeness.

\begin{lemma} \label{lem: entropy bound}
	Let $H_t$ be an $\FF_t$-adapted matrix-valued processes and let $N_t$ be defined by
	$$
	N_t = \int_0^t H_s dB_s.
	$$
	Suppose that $\tilde{H}_t$ is such that for some $t_0 \in [0,1]$:
	\begin{enumerate}
		\item $\tilde{H}_t = H_t$ almost surely, for $t < t_0$.
		\item For $t \geq t_0$, $\tilde{H}_t$ is deterministic  and $\tilde{H}_t \succ \mathrm{I}_d$.
	\end{enumerate}
	Then, if $M_t$ is defined by
	$$
	M_t = \int_0^t \tilde{H}_s dB_s,
	$$
	we have
	$$\mathrm{Ent}\left(N_1||M_1\right) \leq \mathrm{Tr}\int\limits_{t_0}^1\frac{\EE\left[\left(H_t-\tilde{H}_t\right)^2\right]}{1-t}dt. $$
\end{lemma}
\begin{proof}
	Define the process 
	$$Y_t = \int\limits_0^t\tilde{H}_sdB_s + \int\limits_{0}^t\int\limits_{0}^s\frac{H_r - \tilde{H}_r}{1-r}dB_rds.$$
	Denote $u_t = \int\limits_0^t \frac{H_s - \tilde{H}_s}{1-s}dB_s$,
	so that, $dY_t = \tilde{H}_tdB_t + u_tdt$, and, by assumption $u_t = 0$, whenever $t < t_0$. 
	It follows that $Y_t = M_t$ for $t < t_0$ and that, using Fubini's theorem, $Y_1 = N_1$. Indeed, 
	\begin{equation} \label{eq:girsannovfubini}
	Y_1 = \int\limits_0^1\tilde{H}_tdB_t + \int\limits_0^1\int\limits_0^t\frac{H_s-\tilde{H}_s}{1-s}dB_sdt = \int\limits_0^1\tilde{H}_tdB_t + \int\limits_0^1\left(H_t-\tilde{H}_t\right)dB_t =N_1.
	\end{equation}
	We denote now by $P$, the measure under which $B$ is a Brownian motion. If 
	$$
	\mathcal{E} := \exp\left(-\int\limits_0^1\tilde{H}_t^{-1}u_tdB_t - \frac{1}{2}\int\limits_0^1\norm{\tilde{H}_t^{-1}u_t}^2dt\right),
	$$
	and we define the tilted measure $Q = \mathcal{E}P$, then by Girsanov's theorem, $\tilde{B}_t = B_t + \int\limits_0^t \tilde{H}_s^{-1}u_sds$ is a Brownian motion under $Q$ and we have the representation 
	$$Y_t = \int\limits_0^t \tilde{H}_sd\tilde{B}_s.$$
	 If $ t < t_0$, then as $u_t = 0$, we have $\tilde{B}_t = B_t$ and $Y_{t_0}$ has the same law under $Q$ and under $P$, which is the law of $M_{t_0}$. Moreover, for $t \geq t_0$, $\tilde{H}_t$ is deterministic. Therefore, it is also true that the law of
	 $$Y_{t_0} + \int\limits_{t_0}^1 \tilde{H}_td\tilde{B}_t,$$
	 under $Q$ and the law of 
	 $$Y_{t_0} + \int\limits_{t_0}^1 \tilde{H}_tdB_t,$$
	 under $P$ coincide.
	 We thus conclude that, under $Q$, $Y_1$ has the same law as $M_1$ under $P$.
	In particular, if $\rho$ is the density of $Y_1$ with respect to $M_1$, this implies
	$$1 = \EE_P\left[\rho(M_1)\right]=\EE_Q\left[\rho(Y_1)\right]=\EE_P\left[\rho(Y_1)\mathcal{E}\right].$$
	By Jensen's inequality, under $P$,
	$$0 = \ln\left(\EE\left[\rho(Y_1)\mathcal{E}\right]\right) \geq \EE\left[\ln\left(\rho(Y_1)\mathcal{E}\right)\right] = \EE\left[\ln(\rho(Y_1))\right] + \EE\left[\ln(\mathcal{E})\right].$$
	But,
	\begin{align*}
	-\EE\left[\ln(\mathcal{E})\right] &= \frac{1}{2}\int\limits_0^1\EE\left[\norm{\tilde{H}_t^{-1}u_t}^2\right]dt \leq \int \limits_{t_0}^1\EE\left[\norm{u_t}^2\right]dt\\
	& = \mathrm{Tr}\int\limits_{t_0}^1\int\limits_{t_0}^s\frac{\EE\left[\left(H_s - \tilde{H}_s\right)^2\right]}{(1-s)^2}dsdt = \mathrm{Tr}\int\limits_{t_0}^1\int\limits_{s}^1\frac{\EE\left[\left(H_s - \tilde{H}_s\right)^2\right]}{(1-s)^2}dtds\\
	& =  \mathrm{Tr}\int\limits_{t_0}^1\frac{\EE\left[\left(H_s - \tilde{H}_s\right)^2\right]}{1-s}ds,
	\end{align*}
	and, from \eqref{eq:girsannovfubini}
	$$\EE_P\left[\ln(\rho(Y_1))\right] = \mathrm{Ent}(N_1||M_1),$$
	which concludes the proof.
\end{proof}
\begin{remark} \label{rmk:girsanov}
	In order to apply Girsanov's theorem in the proof above, one must also require some integrability condition from the drift $u_t$. It will suffice to assume 
	$$\int\limits_0^1\EE\left[\norm{\tilde{H}_t^{-1}u_t}^2\right]dt < \infty.$$
	Indeed, if $\int\limits_0^1\norm{\tilde{H}_t^{-1}u_t}^2dt$ is uniformly bounded, then Novikov's criterion applies. The general case may then be obtained by an approximation argument (see \cite[Proposition 1]{lehec2013representation} for more details). In our application below this condition will be satisfied.
\end{remark}
We are now in a position to prove that stability with respect to Gaussian mixtures holds in relative entropy.
\begin{proof}[Proof of Theorem \ref{thm: mixture stability}]
	Fix $t_0\in[0,1]$, by \eqref{eq: deltatal} we get
	\begin{equation} \label{eq: truncated delta}
	\delta_{\mathrm{Tal}}(\mu) \geq t_0\mathrm{Tr}\int\limits_{t_0}^1\frac{\EE\left[\left(\mathrm{I}_d-\Gamma_t\right)^2\right]}{1-t}dt.
	\end{equation}
	Define the matrix-valued process
	$$\tilde{\Gamma}_t = \begin{cases}
	\Gamma_t& 0 \leq t < t_0\\
	\frac{1 - t_0}{t_0(t-2)+ 1}\mathrm{I}_d& t_0 \leq t \leq 1\\
	\end{cases},$$
	and the martingale
	$$M_t = \int\limits_{0}^t\tilde{\Gamma}_sdB_s.$$
	One may verify that
	$$\int\limits_{t_0}^1\left(\frac{1-t_0}{t_0(t-2) + 1}\right)^2dt = 1,$$
	which implies, $M_1 - M_{t_0} = \int\limits_{t_0}^1\tilde{\Gamma}_t(M_t)dB_t \sim \gamma.$ Also, from \eqref{eq: gammadef},
	$$M_t = \int\limits_{0}^{t_0}\tilde{\Gamma}_tdB_t = \int\limits_{0}^{t_0}\Gamma_tdB_t= \EE\left[X_1|\mathcal{F}_{t_0}\right].$$
	If $\nu_{t_0}$ is the law of $\EE\left[X_1|\mathcal{F}_{t_0}\right]$, then since $\{B_s\}_{s>t_0}$ is independent from $\EE\left[X_1|\mathcal{F}_{t_0}\right]$, we have that $\nu_{t_0}*\gamma$ is the law of $M_1$.\\
	
	 We now invoke Lemma \ref{lem: entropy bound} with the process $\EE\left[X_1|\FF_t\right]$ as $N_t$. Since $\tilde{\Gamma}_t$ meets the conditions of the lemma, we get
	\begin{align*}
	\mathrm{D}(X_1||M_1) = \mathrm{D}(\mu||\nu_{t_0}*\gamma) &\leq \mathrm{Tr}\int\limits_{t_0}^1\frac{\EE\left[\left(\Gamma_t-\tilde{\Gamma}_t\right)^2\right]}{1-t}dt \\
	 &\leq 2\mathrm{Tr}\int\limits_{t_0}^1\frac{\EE\bigg[\Big(\Gamma_t-\mathrm{I}_d\Big)^2\bigg]}{1-t}dt + 2\mathrm{Tr}\int\limits_{t_0}^1\frac{\EE\left[\left(\tilde{\Gamma}_t-\mathrm{I}_d\right)^2\right]}{1-t}dt.
	\end{align*}	
	Observe that by showing that the above integrals are finite we will also verify the integrability condition from Remark \ref{rmk:girsanov}. Applying \eqref{eq: truncated delta},
	$$2\mathrm{Tr}\int\limits_{t_0}^1\frac{\EE\bigg[\Big(\Gamma_t-\mathrm{I}_d\Big)^2\bigg]}{1-t}dt \leq 2\frac{\delta_{\mathrm{Tal}}(\mu)}{t_0}.$$
	To bound the second term we calculate 
	\begin{align*}
	2\mathrm{Tr}\int\limits_{t_0}^1\frac{\EE\left[\left(\tilde{\Gamma}_t-\mathrm{I}_d\right)^2\right]}{1-t}dt &= 2d\int\limits_{t_0}^1\frac{\left(\frac{1 - {t_0}}{{t_0}(t-2)+ 1} - 1\right)^2}{1-t}dt\\
	 &=2d\left(-\ln(1 +{t_0}(t-2))-\frac{1-{t_0}}{2(t-{t_0}) + 1}\right)\Big\vert_{t_0}^1\\
	 &=2d\left(\ln(1-{t_0}) + \frac{{t_0}}{1-{t_0}}\right).
	\end{align*}
	Combining the last displays, we get
	$$\mathrm{D}(\mu||\nu_{t_0}*\gamma) \leq2\left(\frac{\delta_{\mathrm{Tal}}(\mu)}{t_0} + d\left(\ln(1-{t_0}) + \frac{{t_0}}{1-{t_0}}\right)\right).$$
	Suppose that $\delta_{\mathrm{Tal}}(\mu) \geq d$, then choosing ${t_0} = \frac{1}{2}$ gives
	$$\frac{\mathrm{D}(\mu||\nu_{t_0}*\gamma)}{6}  \leq  \delta_{\mathrm{Tal}}(\mu).$$
	Otherwise, $\delta_{\mathrm{Tal}}(\mu) < d$ and we choose ${t_0} = \left(\frac{\delta_{\mathrm{Tal}}(\mu)}{d}\right)^{\frac{1}{3}} \leq \frac{1}{2}$. A second order approximation, shows that for $s \in[0,\frac{1}{2}]$,
	$$\ln(1-s) + \frac{s}{1-s} \leq 2s^2.$$
	Hence, for the above choice of ${t_0}$,
	$$\mathrm{D}(\mu||\nu_{t_0}*\gamma) \leq2\frac{\delta_{\mathrm{Tal}}(\mu)}{t_0} +4d{t_0}^2 = 3\delta_{\mathrm{Tal}}(\mu)^{\frac{2}{3}}d^{\frac{1}{3}}.$$
	This implies
	 $$\frac{1}{3\sqrt{3}}\frac{\mathrm{D}(\mu||\nu_{t_0}*\gamma)^{\frac{3}{2}}}{\sqrt{d}}\leq \delta_{\mathrm{Tal}}(\mu),$$
	which is the desired claim.
	Finally, by the law of total variance, it is immediate that
	$$\mathrm{Cov}\left(\nu_{t_0}\right)\preceq\mathrm{Cov}\left(\mu\right).$$
\end{proof}
\section{An application to Gaussian concentration} \label{sec: application}
We now show that our stability bounds imply an improved Gaussian concentration inequality for concave functions. 
\begin{corollary} \label{cor: concntration}
	Let $f$ be a concave function and $G \sim \gamma$ in $\RR^d$. Suppose that $f$ is even, then for any $t \geq0$,
	$$\PP\left(f(G)\geq t\right) \leq e^{-\frac{4t^2}{7\EE\left[\norm{\nabla f(G)}_2^2\right]}}.$$
\end{corollary}
Before proving the result we mention that our proof follows the one presented in \cite{samson2000concentration}. We use Theorem \ref{thm: poincare stability} to improve the constant obtained there. One should also compare the corollary to the main result of \cite{paouris2018gaussian} which used Ehrhard's inequality in order to show that $\EE\left[\norm{\nabla f(G)}_2^2\right]$ may be replaced by the smaller quantity $\mathrm{Var}(f(G))$, at the cost of a worse constant in the exponent.\\
The assumption that $f$ is even is used here for simplicity and could be relaxed.
\begin{proof}[Proof of Corollary \ref{cor: concntration}]
	For $\lambda >0 $, denote the measure $\nu_\lambda = \frac{e^{\lambda f}}{\EE_\gamma\left[e^{\lambda f}\right]}d\gamma$ and
	let $(X,Y)$ be a random vector in $\RR^{2d}$ which is a realization of the optimal coupling between $\nu_\lambda$ and $\gamma$. That is, 
	$X \sim \nu_{\lambda}, Y \sim \gamma$ and
	$$\mathcal{W}_2(\nu_\lambda,\gamma) = \sqrt{\EE\left[\norm{X-Y}_2^2\right]}.$$
	As $f$ is concave, we have by using Cauchy-Schwartz:
	\begin{align} \label{eq: coupling bound}
	\EE_{\nu_{\lambda}}\left[\lambda f\right] - \EE_\gamma\left[\lambda f\right] &\leq \EE\left[\left\langle\nabla \lambda f(Y),X-Y\right\rangle\right]\leq\sqrt{\lambda^2 \EE\left[\norm{\nabla f(Y)}_2^2\right]}\sqrt{\EE\left[\norm{X-Y}_2^2\right]} \nonumber\\
	&=\sqrt{ \lambda^2 \EE_\gamma\left[\norm{\nabla f}_2^2\right]}\mathcal{W}_2(\nu_\lambda,\gamma).
	\end{align}
	Since $f$ is concave, $\nu_\lambda$ has a log-concave density with respect to the standard Gaussian. For such measures, Brascamp-Lieb's inequality (\cite{brascamp1976extensions}) dictates that $\mathrm{C_p}(\nu_\lambda)\leq 1$. Note that
	$$\frac{(x+1)(2-2x+(x+1)\ln(x)}{(x-1)^3}\geq \frac{1}{3}, \text{ whenever } x \in[0,1].$$
	In this case, since $f$ is even and $\nu_{\lambda}$ is centered, Theorem \ref{thm: poincare stability} gives us, 
	$$\delta_{\mathrm{Tal}}(\nu_\lambda)\geq \frac{1}{4}\mathrm{D}\left(\nu_\lambda||\gamma\right),$$
	which is equivalent to 
	$$\mathcal{W}_2^2\left(\nu_\lambda,\gamma\right)\leq \frac{7}{4}\mathrm{D}(\nu_\lambda||\gamma).$$ 
	Combining this with \eqref{eq: coupling bound} and the assumption, $\EE_\gamma\left[\lambda f\right] =0 $, yields 
	$$\EE_{\nu_{\lambda}}\left[\lambda f\right]\leq \sqrt{\lambda^2\frac{7}{4}\EE_\gamma\left[\norm{\nabla f}_2^2\right]\mathrm{D}(\nu_\lambda||\gamma)}.$$
	For any $x,y \geq 0$ we have the inequality,
	$$\sqrt{xy} \leq \frac{x}{4} + y.$$
	Observe as well that 
	$$\mathrm{D}(\nu_\lambda||\gamma) =  \EE_{\nu_{\lambda}}\left[\lambda f\right]  -\ln\left(\EE_\gamma\left[e^{\lambda f}\right]\right).$$
	Thus,
	$$\ln\left(\EE_\gamma\left[e^{\lambda f}\right]\right)\leq \lambda^2\frac{7}{16}\EE_\gamma\left[\norm{\nabla f}_2^2\right].$$
	By Markov's inequality, for any $\lambda, t > 0$
	$$\PP\left(f(G) \geq t\right) = \PP\left(e^{\lambda f(G)} \geq e^{\lambda t}\right) \leq \EE_\gamma\left[e^{\lambda f }\right]e^{-\lambda t} \leq \exp\left(\lambda^2\frac{7}{16}\EE_\gamma\left[\norm{\nabla f}_2^2\right] -\lambda t\right).$$
	We now optimize over $\lambda$ to obtain,
	$$\PP\left(f(G) \geq t\right) \leq e^{-\frac{4t^2}{7\EE_\gamma\left[\norm{\nabla f}_2^2\right]}}.$$
\end{proof}
\bibliographystyle{plain}
\bibliography{bib}{}
\end{document}